\documentclass[12pt]{amsart}
\usepackage{fullpage}
\usepackage{graphicx}
\usepackage{amsfonts}
\usepackage{amstext}
\usepackage{amssymb}
\usepackage{amsmath}
\usepackage{amsthm}
\usepackage{amsopn}
\usepackage{bm}
\usepackage[ps,dvips,matrix,curve,frame,arrow,rotate,line]{xy}

\newtheorem{theorem}{Theorem}
\newtheorem{prop}[theorem]{Proposition}
\newtheorem{lemma}[theorem]{Lemma}
\newtheorem{corollary}[theorem]{Corollary}

\theoremstyle{remark}

\theoremstyle{definition}
\newtheorem{remark}[theorem]{Remark}
\newtheorem{example}[theorem]{Example}
\newtheorem{defn}[theorem]{Definition}

\newcommand{\divides}{\mid}

\newcommand{\C}{\mathcal{C}}

\newcommand{\OO}{\mathcal{O}}
\newcommand{\VV}{\mathcal{V}}
\newcommand{\QQ}{\mathbb{Q}}
\newcommand{\RR}{\mathbb{R}}
\newcommand{\ZZ}{\mathbb{Z}}
\newcommand{\NN}{\mathbb{N}}

\newcommand{\FF}{\mathbb{F}}
\newcommand{\isom}{\cong}

\newcommand{\pp}{\mathfrak{p}}

\newcommand{\ol}[1]{\overline{#1}}
\newcommand{\oh}[1]{\widehat{#1}}
\newcommand{\sq}[1]{\widetilde{#1}}

\DeclareMathOperator{\res}{res}

\DeclareMathOperator{\Char}{char}

\newcommand{\Card}[1]{\##1}

\newcommand{\eqdef}{\overset{\text{def}}{=}}
\newcommand{\ilim}{\mathop{\varprojlim}\limits}

\swapnumbers
\numberwithin{theorem}{section}

\newcommand{\maM}{locally scaling}
\newcommand{\cip}{corresponding interpolating polynomials}
\newcommand{\TBRBO}{T.B.R.B.O.}
\renewcommand{\v}{w}

\begin{document}
\title{On measure-preserving $\C^1$ transformations of compact-open subsets of non-archimedean local fields}
\author{James Kingsbery}
\address[James Kingsbery]{Department of Mathematics\\
     Williams College \\ Williamstown, MA 01267, USA}
\email{James.C.Kingsbery@williams.edu}
\author{Alex Levin}
\address[Alex Levin]{Harvard University, MA 02138, USA}
\email{levin@fas.harvard.edu}
\author{Anatoly Preygel}
\address[Anatoly Preygel]{Harvard University, MA 02138, USA}
\email{preygel@post.harvard.edu}
\author{Cesar E. Silva}
\address[Cesar Silva]{Department of Mathematics\\
     Williams College \\ Williamstown, MA 01267, USA}
\email{csilva@williams.edu}
\begin{abstract}
We introduce the notion of a \emph{\maM} transformation defined on a compact-open subset of a non-archimedean local field.  We show that this class encompasses the Haar measure-preserving transformations defined by $\C^1$ (in particular, polynomial)  maps, and prove a structure theorem for {\maM} transformations.  We use the theory of polynomial approximation on compact-open subsets of non-archimedean local fields to demonstrate the existence of ergodic Markov, and mixing Markov transformations defined by such polynomial maps.  We also give simple sufficient conditions on the Mahler expansion of a continuous map $\ZZ_p \to \ZZ_p$ for it to define a Bernoulli transformation.
\end{abstract}
\maketitle

\section{Introduction}\label{sec:intro}
The $p$-adic numbers have arisen  in a natural way in the study of some 
dynamical systems, for example in the study of group automorphisms of 
solenoids in Lind and Schmidt \cite{LindSchmidt}; other situations in dynamics where the 
$p$-adic numbers come up are surveyed in Ward \cite{Ward}.  At the same time 
there has been interest in studying the dynamics (topological, 
complex,  or measurable) of naturally arising maps (such as 
polynomials) defined on  the $p$-adics; see for
example Benedetto \cite{Benedetto},  Khrennikov and Nilson \cite{KhrennikovNilson}, and Rivera-Letelier \cite{RiveraL}.  In particular, 
Bryk and Silva in \cite{BrykSilva} studied the measurable dynamics of 
simple polynomials on balls and spheres
on the field $\QQ_p$ of  $p$-adic numbers.  The maps they studied are 
ergodic but not totally ergodic and they asked whether there exist 
polynomials on $\QQ_p$ that define (Haar) measure-preserving 
transformations that are mixing.   Woodcock and Smart in \cite{WoodcockSmart} show that the polynomial map $x \mapsto \tfrac{x^p-x}{p}$ defines a Bernoulli, hence mixing,  transformation on $\ZZ_p$.  A  consequence of our work is a significant extension of the result for this map, placing it in a greater context (see in particular Example~\ref{ex:x^p-x}).

Rather than working on $\QQ_p$ we find that the natural setting for our work is over a non-archimedean local field $K$.  We introduce a class of transformations, called {\maM}, and show in Lemmas~\ref{lem:C1Deriv} that measure-preserving $\C^1$ (in particular, polynomial)  maps are {\maM}.  In Section~\ref{sec:structMaM} we apply the theory of Markov shifts to classify the dynamics of {\maM} transformations, decomposing the transformation into a disjoint union of ergodic Markov transformations and local isometries.  In particular, we show that a weakly mixing {\maM} transformation must be mixing.  We also show the existence of polynomials defining transformations exhibiting nearly the full range of behaviors possible for {\maM} transformations, such as ergodic Markov, mixing Markov, and Bernoulli transformations.  

Given a polynomial defined on a compact-open subset of $K$, our work shows that a finite computation may check whether it defines a measure-preserving transformation and whether it defines a mixing transformation; the question of ergodicity is also answered, except in the case where the polynomial is $1$-Lipschitz, which has been studied by Anashin in \cite{Anashin}.

We briefly mention related works studying measurable dynamics of certain
maps on spaces related to the $p$-adics.  These  works \cite{RumelyBaker},
\cite{FavreRL04}, and \cite{FavreRL} construct a natural
invariant measure for a wide-class of rational functions, as in existing
constructions in complex dynamics.  The natural domain for these
constructions is the so-called Berkovich projective space, a space much
larger than the ordinary $p$-adics.

We now indicate an outline of the rest of the paper.  Section~\ref{sec:markov} reviews results on Markov shifts, Section~\ref{sec:analyticNotation} reviews preliminaries on non-archimedean local fields as well as some analytic definitions, and Section~\ref{sec:polApprox} recalls some of the theory of polynomial approximation on rings of integers of non-archimedean local fields.  

Section~\ref{sec:mpC1} establishes the fact that measure-preserving $\C^1$ maps are {\maM}, and Section~\ref{sec:structMaM} proves our main structural results, in particular Proposition~\ref{prop:maMPhi} and Theorem~\ref{thm:maMMPStruct}.  Section~\ref{sec:polMaM}, in particular Theorem~\ref{thm:polyRepMaM}, shows that polynomial maps are in a sense a representative class of {\maM} transformations, and demonstrates the existence of polynomial maps defining {\maM} transformation with various behaviors, including mixing.   Section~\ref{sec:polBern} and Section~\ref{sec:polAlmostBern} are devoted to demonstrating two interesting classes of {\maM} maps on $\ZZ_p$ that arise naturally in the study of polynomial approximations. Specifically, Section~\ref{sec:polBern} studies maps which are isometrically conjugate to the natural realization of the (one-sided) Bernoulli shift, and shows for instance that the map $x \mapsto {x \choose p^\ell}$ on $\ZZ_p$ is Bernoulli.  Section~\ref{sec:polAlmostBern} then studies similar binomial-coefficient maps which are {\maM} and so have very regular structures but fail to be Haar measure-preserving. 

\subsection{Acknowledgements}
This paper is based on research by the Ergodic Theory group of the 2005 SMALL  summer research project at Williams College.  Support for the project was provided by National Science Foundation REU Grant DMS - 0353634 and the Bronfman Science Center of Williams College.  The authors would like to thank several anonymous referees for careful readings of the paper and  valuable suggestions. 

\section{Markov shifts}\label{sec:markov}
Let $H$ be a finite non-empty set.  By a \emph{stochastic matrix on $H$} we mean a map $A: H^2 \to \RR_{\geq 0}$ such that \[ \sum_{j \in H} A(i,j) = 1 \qquad \text{for each $i \in H$}. \]  Putting $H$ into a bijection with the set $\{0,\ldots,\Card{H}-1\}$ we may regard $A$ as a $\Card{H} \times \Card{H}$ matrix with non-negative entries and the entries in each row summing to $1$.  In analogy with this case, we will refer to the sets $\{ A(i,\cdot) \}$ and $\{A(\cdot,j)\}$ as \emph{rows} and \emph{columns of $A$}, respectively.

By a \emph{row vector on $H$} we mean a map $\v: H \to \RR$.  For $\v$ a row vector and $A$ a stochastic matrix, we define their product as the row vector $\v A$ defined by
\[ \v A(j) = \sum_{i \in H} \v(i) A(i,j). \]  We will say that $\v$ is non-negative (resp. positive) if it takes values in $\RR_{\geq 0}$ (resp $\RR_{>0}$).

To any stochastic matrix $A$ we may associate the following symbolic dynamical system:
\begin{enumerate}
\item Let \[ X_A = \{ x \in \prod_{i \geq 0} H: A(\pi_{n}(x),\pi_{n+1}(x)) \neq 0 \text{ for all } n \geq 0 \} \] where $\pi_{n}: \prod_{i \geq 0} H \to H$ is projection to the $n^\text{th}$ coordinate. Give each finite factor the discrete topology, and $X_A$ the subspace topology inherited from the product topolog
\item Let $T_A: X_A \to X_A$ be defined by $\pi_n \circ T_A = \pi_{n+1}$; that is, $T_A$ is simply ``shifting left.''  Then, $(X_A,T_A)$ is a topological dynamical system.
\item If in addition we are given a non-negative row vector $\v$, then we may define a measure on $X_A$ by
\[ \mu_{A,\v}([d_0 d_1 d_2 \ldots d_\ell]) = \v(d_0) A(d_0,d_1) \cdots A(d_{\ell-1} d_\ell) \text{, where }[d_0 \ldots d_\ell] \eqdef \bigcap_{n=0}^{\ell} \pi_{n}^{-1}(d_n). \]
We call a set of the form $[d_0 \ldots d_\ell]$ a \emph{cylinder set}; we may observe that the cylinder sets form a base for the topology on $X_A$.  Note that if $\v$ is in fact positive, then $\mu_{A,\v}$ assigns positive measure to each cylinder set and hence to each open set.  We may check that if $\v = \v A$ then $T_A$ is measure-preserving with respect to $\mu_{A,\v}$. 
\end{enumerate}

We call such a dynamical system a \emph{Markov shift}.  We say that a dynamical system is \emph{Markov} if it is isomorphic to some Markov shift.

We say that a stochastic matrix $A$ is \emph{irreducible} or \emph{ergodic} if for each $i,j \in H$ there exists a $n \in \NN$ such that $A^n(i,j) > 0$.  This condition has a natural interpretation in terms of the connectedness of a certain directed graph associated with $A$, as we shall see in the proof of Proposition~\ref{prop:ergDecMarkov}.  We say that a stochastic matrix $A$ is \emph{primitive} if there exists a $n \in \NN$ such that $A^n(i,j) > 0$ for all $i,j \in H$. 

Using the Perron-Frobenius Theorem on non-negative irreducible and primitive matrices, along with a graph theoretic interpretation of the stochastic matrix, one may  obtain an ergodic decomposition result for Markov shifts:
\begin{prop}\label{prop:ergDecMarkov}
Let $A$ be a stochastic matrix, and $\v$ a positive row vector such that $\v = \v A$.

Then, we may partition $H$ into disjoint sets \[ H = \bigsqcup_{k=1}^{n} H_k \] such that
\begin{enumerate}
\item $A(i,j)=0$ for $i \in H_k, j \in H_\ell$ with $k \neq \ell$; and
\item $A_k = \left.A\right|_{H_k \times H_k}$ is irreducible for $k=1,\ldots,n$
\end{enumerate}
Then, $\v_k = \left. \v \right|_{H_k}$ satisfies $\v_k = \v_k A_k$. And we have the ergodic decomposition of $(X_A,\mu_{A,\v},T_A)$ as
\[ (X_A,\mu_{A,\v},T_A) = \bigsqcup_{k=1}^{n} (X_{A_k},\mu_{A_k,\v_k},T_{A_k}), \] where $X_{A_k}$ is viewed as a $T_A$-invariant subset of $X_A$, so that $\mu_{A_k,\v_k} = \left. \mu_{A,\v} \right|_{A_k}$ and $T_{A_k} = \left. T_{A} \right|_{A_k}$.

Moreover, the $k^\text{th}$ summand is mixing if $A_k$ is primitive.
\end{prop}
\begin{proof}
Construct a graph on $H$ as follows.  We place a directed edge from $i \to j$ if and only if $A(i,j) > 0$.  As $\v$ is strictly positive, this is equivalent to the condition that $\v(i) A(i,j) > 0$.  We say that the flow or flux associated to this edge is $\v(i) A(i,j)$.  Now, the flow out of $i$ is \[ \sum_{j \in H} \v(i) A(i,j) = \v(i) \] as $A$ is a stochastic matrix.  The flow into $i$ is \[ \sum_{k \in H} \v(k) A(k,i) = \v(i) \] as $\v = \v A$.  So, we see that the flux into and out of $i$ are both equal to $\v(i)$.  

This implies that for every finite subset of $H$, the in-flux and out-flux will be equal.  For $i \in H$, let $R(i)$ be the set of points reachable from $i$, and $B(i)$ the set of points which can reach $i$.  Note that $R(i)$ has out-flux $0$ by construction, and $B(i)$ has in-flux $0$ by construction; as $H$ is finite, these subsets are finite, so both have in-flux and out-flux equal to $0$.  

Now, we can have no edges into or out of either of these two sets.  But, if $t \in R(i)$ and $y \in B(i)$, then there is a path from $y$ to $t$; so we must have $t \in B(i)$ and $y \in R(i)$, and so $B(i)=R(i)$.  So, $B(i)=R(i)$ is strongly connected, and there are no edges into or out of this set.

For $\ell>0$, note that $A^\ell(i,j) > 0$ is equivalent to there being a path of length precisely $\ell$ from $i$ to $j$.  It follows that the collection \[ \{ B(i) : i \in H \} \] gives our desired decomposition of $H$.

%
We readily note that $\v_k = \v_k A_k$ for for $k=1,\ldots,n$.  Then, as $A_k$ is irreducible, \cite[Theorem 1.19]{Walters} implies that the $k^\text{th}$ summand is ergodic, from which the ergodic decomposition follows.  Finally, \cite[Theorem 1.31]{Walters} implies that the $k^\text{th}$ summand is mixing if $A_k$ is primitive.
\end{proof}

\section{Analytic definitions, preliminaries, and notation}\label{sec:analyticNotation}
Let $K$ be a \emph{non-archimedean local field}, which we take to be either a finite field extension of $\QQ_p$ or $\FF_{p^n}((t))$ for some prime $p$.

Let $|\cdot|$ be a non-archimedean multiplicative valuation (sometimes called a ``non-archimedean absolute value'') on $K$, such that $|\cdot|$ generates the topology on $K$.  Denote $\VV = |K^\times|= \{ |x| : x \in K^\times\}$, $\OO = \{ x \in K : |x| \leq 1 \}$ and $\pp = \{ x \in K : |x| < 1 \}$ (note that $\OO, \pp$ are independent of the choice of valuation).   It is the case that $\OO$ is a ring with maximal idea $\pp$ and that $\OO/\pp$ is a finite field (the \emph{residue field}).  Let $p = \Char \OO/\pp$, $q= \Card{\OO/\pp}$, both finite with $q$ a power of $p$.  

We denote \[ B_r(x) = \{ y \in K : |x-y| \leq r\}, \] and call such a set (for any value of $r$) a \emph{ball}.  A ball of radius precisely $r$ will be called an \emph{$r$-ball}.  Let $\mu$ be Haar measure on $K$, normalized such that $\mu(\OO)=1$; define $\rho: \VV \to \RR_{>0}$ by $\rho(r) = \mu(B_r(0))$.

Now, we recall the following standard results:
\begin{enumerate}
\item $\OO$ is a discrete valuation ring with unique maximal ideal $\pp$;
\item $\pp = \pi \OO$ for any $\pi \in \pp \setminus \pp^2$; we call any such $\pi$ a \emph{uniformizing parameter}; 
\item $\VV$ is the discrete abelian (multiplicative) subgroup of $\QQ$ generated by $|\pi|$;  in light of this, we may define a map $v: K \to \ZZ \cup \{+\infty\}$ defined by $v(0) = +\infty$ and $v(x) = \log_{|\pi|} |x|$ for $x \in K^\times$; this is the additive valuation (sometimes just ``valuation'') on $K$;
\item For $r = |\pi|^k$, $k \geq 0$ it is the case that
\[ \rho(r) = \mu(B_r(0)) = \left(\Card{\OO/\pp^k}\right)^{-1} = q^{-k}. \]
Indeed, for $r \in \VV$ we see that $\rho(r) = q^{-\log_{|\pi|} r}$;
\item A subset $X \subseteq K$ is compact-open if and only if $X$ is a finite union of balls.
\end{enumerate}
We direct the interested reader to \cite{Serre} for a thorough treatment of related topics.

We will continue to use the symbols $K, \mu, p, q, |\cdot|, \OO, \pp, \pi, v, \VV, \rho, B_r$ with these meanings below.

Let $X$ be an open subset of $K$ and $a \in X$.  Then, we say that a function $f: X \to K$ is \emph{strictly differentiable} or \emph{$\C^1$} at $a$ (denoted $f \in \C^1(a)$) if the limit \[ \lim_{(x,y) \to (a,a)\atop{x \neq y}} \frac{f(x)-f(y)}{x-y} \] exists.  We write $f \in \C^1(X)$ if $f \in \C^1(a)$ for each $a \in X$.  For more on this notion, see \cite{Schikhof} or \cite{Robert}.


\section{Measure preserving $\C^1$ maps on non-archimedean local fields}\label{sec:mpC1}
\begin{defn}\label{defn:maM}
For $X \subseteq K$ compact-open, we say that a transformation $T: X \to X$ is \emph{\maM} for $r \in \VV$ if $X$ is a finite union of $r$-balls and if there exists a function $C: X \to \RR_{\geq 1}$ such that
\[ |x-y| \leq r \Rightarrow |T(x)-T(y)| = C(x) |x-y| \quad \text{for $x,y \in X$}. \]  We will refer to $C$ as the \emph{scaling function}. 
\end{defn}
\begin{remark}
Let us make the following observations about {\maM} transormations:
\begin{enumerate}
\item By the symmetry of $x$ and $y$ in the previous displayed equation, $C$ is constant on cosets of $B_r(0)$.  We will write $H = X/B_r(0)$ for the set of cosets of $B_r(0)$ contained in $X$ (recall that $X$ is a union of such cosets); we treat elements of $H$ as subsets of $X$.  Then, $C$ induces a map $C: H \to \RR_{\geq 1}$.
\item The terminology ``{\maM}'' is convenient but perhaps slightly misleading: For us, such transformations must not only locally scale distances, but must do so by a factor that is at least $1$.
\end{enumerate}
\end{remark}

This definition is motivated by the ease of analyzing the structure of such maps together with the following easy lemma:
\begin{lemma}\label{lem:C1Scaling}
Let $X \subset K$ be open, and suppose $f \in \C^1(X)$ is such that $f'(x) \neq 0$ for $x \in X$.  Then, $|f'(x)|$ is locally constant on $X$.  If, moreover, $X$ is compact, $f(X) \subset X$, and $|f'(a)| \geq 1$ or all $a \in X$, then the induced transformation $f: X \to X$ is {\maM} for some $r \in \VV$.
\end{lemma}
\begin{proof} 
Fix $a \in X$. Since $X$ is open and $f \in \C^1(a)$, there exists $r_a \in \VV$ so that $B_{r_a}(a) \subset X$ and
\[ \left| \frac{f(x) - f(y)}{x-y} - f'(a)\right| < |f'(a)|  \qquad \text{for $x, y \in B_{r_a}(a)$, $x \neq y$}. \]  By the strong triangle inequality, it follows that $|f(x) -f(y)| = |f'(a)| |x-y|$ for all $x, y \in B_r(a)$.  In particular, we note that $|f'(x)|$ is constant on $B_{r_a}(a)$.

If $X$ is compact, then there is a finite set $a_0, \ldots, a_k \in X$ so that $B_{r_{a_0}}(a_0), \ldots, B_{r_{a_k}}(a_k)$ is an open cover of $X$.  Taking $r \leq \min_{i=0}^{k} r_{a_k}$, we see that $X$ is a union of $r$-balls and that $f$ restricted to each $B_r(a)$ scales distance by $|f'(a)| \geq 1$  by the previous displayed equation.  So, $f$ is {\maM} for $r \in \VV$ with scaling function $C(x) = |f'(x)|$.
\end{proof}

\begin{lemma}\label{lem:C1Deriv}
Let $X \subseteq K$ be open.  Let $f \in \C^1(X)$ be such that $f(X) \subseteq X$ and such that the transformation $f: X \to X$ is measure-preserving with respect to $\res{\mu}{X}$.  Then, $|f'(a)| \geq 1$ for all $a \in X$.  Furthermore, $f$ is {\maM} for some $r \in \VV$.
\end{lemma}
\begin{proof}
Suppose there is an $a \in X$ with $|f'(a)| < 1$.  Take $\alpha \in \VV$ such that $|f'(a)| \leq \alpha < 1$.  As $X$ is open, there exists $r' \in \VV$ such that $B_{r'}(a) \subseteq X$ and $B_{r'}(f(a)) \subseteq X$.  Moreover, as $f \in \C^1(a)$ we may take $r \in \VV$, with $r \leq r'$, such that 
\[ \left| f(x)-f(y) - f'(a)(x-y)\right| \leq \alpha\left|x-y\right| \quad\text{for any $x,y \in B_r(a)$.} \] It follows, by the strong triangle inequality, that $\left|f(x) - f(y)\right| \leq |x-y| \alpha$ for $x,y \in B_r(a)$.  So, $B_{\alpha r}(f(a)) \subseteq X$ by construction and moreover $f^{-1}\left(B_{\alpha r}(f(a)) \right) \supseteq B_r(a)$.  Taking measures we note that
\[\mu\left(B_{\alpha r}(f(a))\right) = \rho(\alpha r) < \rho(r) = \mu\left(B_r(a)\right) \leq \mu\left(f^{-1}\left(B_{\alpha r}(f(a))\right)\right).\]  So, $f$ is not measure-preserving.

This proves that $|f'(a) \geq 1$ for all $a \in X$.  The remaining part of the claim follows from Lemma~\ref{lem:C1Scaling}.
\end{proof}

\begin{remark}\label{rem:polyC1}
Note that for $f \in K[x]$, $f \in \C^1(K)$.  So, if $f(X) \subset X$ induces a measure-preserving transformation $f: X \to X$, then Lemma~\ref{lem:C1Deriv} and Lemma~\ref{lem:C1Scaling} imply that $f$ is {\maM}.
\end{remark}

%
%
%
%
\begin{example}\label{ex:Choose2}
Consider the map $f: \ZZ_2 \to \ZZ_2$ defined by \[ f(x) = {x \choose 2} = \frac{x(x-1)}{2}.\]  Then, \[ |f(x)-f(y)|=\left\lvert\frac{(x-y)(x+y-1)}{2}\right\rvert=2|x-y| |x+y-1| \]  So, $f$ is {\maM} for $r = 1/2$, since $|x+y-1|=1$ when $|x-y| \leq 1/2$.  We will see in Section~\ref{sec:polBern} that $f:\ZZ_2 \to \ZZ_2$ is actually measure-preserving and in fact Bernoulli.
\end{example}

\begin{remark}\label{rem:JK2} Suppose $f \in K[x]$ defines a transformation $f: \ZZ_p \to \ZZ_p$.  Since $f$ is polynomial, we have $f \in \C^1(\ZZ_p)$ and thus, by Lemma~\ref{lem:C1Scaling}, is {\maM} for some $r \in \VV$.  We can use the Taylor expansion of $f$ to find such an $r$ (this idea is similar to that in \cite[p. 33, Lemma 1.6]{KhrennikovNilson}).  Specifically, writing
\[ f(x+z) - f(x) = z f'(x) + \sum_{k=2}^{\deg f} \frac{z^k}{k!} f^{(k)}(x) \] we note by the strong triangle inequality that it suffices to choose $r$ so that \[ \left| \frac{f^{(k)}(x)}{k!} \right| r^{k-1} < |f'(x)| \]  for all $x \in \ZZ_p$ and $n \geq 2$.
\end{remark}

%
%
%

\section{Structure of {\maM} transformations}\label{sec:structMaM}
\begin{lemma}\label{lem:maMScale}
Let $X \subseteq K$ be compact-open, with $T: X \to X$ {\maM} for $r \in \VV$.  Let $C: X \to \RR_{\geq 1}$ be the scaling function of Definition~\ref{defn:maM}.  Then, for each $a \in X$ and $r' \in \VV$ with $r' \leq r$, the map
\[ \left. T \right|_{B_{r'}(a)}: B_{r'}(a) \to B_{r' C(a)}(T(a)) \] is a bijection.
\end{lemma}
\begin{proof}
Denote $B = B_{r'}(a)$ and $B' = B_{r' C(a)}(T(a))$.  As $C$ is constant on $B_r(a)$, $|Tx-Ty|=C(x) |x-y| = C(a) |x-y|$ for all $x,y \in B \subseteq B_r(a)$.  This implies that $T(B) \subseteq B'$, so our restriction is well-defined.  It also implies that the restriction is injective.

For each $k \geq 0$ we may take coset representatives $a_0,\ldots,a_{q^k-1}$ for $B/B_{r' |\pi|^k}(0)$. Then for $i,j \in \{ 0,\ldots,q^{k}-1\}$ we have 
\[ |T(a_i)-T(a_j)| = C(a) |a_i-a_j| > r' C(a) |\pi|^k. \]  So, $T(a_0),\ldots,T(a_{q^k-1})$ are precisely the $q^k$ coset representatives for $B'/B_{r' C(a) |\pi|^k}(0)$.  It follows that $T(B)$ is dense in $B'$.

Now, note that $\left. T \right|_{B}$ is continuous.  So, $T(B)$ is the continuous image of a compact set, thus compact, and so closed.  So, $T(B)=B'$.  This proves surjectivity, and the lemma is proved.
\end{proof}

\begin{corollary}\label{cor:maMScale}
Let $X, T, C$ be as in Lemma~\ref{lem:maMScale}.  Set $H=X/B_r(0)$.  For any $i,j \in H$, $a \in j$, and $r' \in \VV$ with $r' \leq r$, the set
\[ i \cap T^{-1}\left(B_{r'}(a) \right) \]
is either the empty set or a ball of radius $r'/C(i)$, according as whether $i \cap T^{-1}(j)$ is empty or not.
\end{corollary}
\begin{proof}
Denote $B = B_{r'}(a)$.  Assume $i \cap T^{-1}(j)$ is not empty, so there is a $y \in i \cap T^{-1}(j)$.
Then, the map \[ \left. T \right|_{B_r(y)}: i= B_r(y) \to B_{r C(i)}(T(y)) \supseteq j \supseteq B \] is a bijection.  It follows that $i \cap T^{-1}(B)$ is non-empty, and we may in fact assume that $y \in i \cap T^{-1}(B)$.

Then, as
\[ \left. T \right|_{B_{r'/C(i)}(y)}: B_{r'/C(i)}(y) \to B_{r'}(T(y)) = B \]
is also a bijection, it follows that $i \cap T^{-1}(B) = B_{r'/C(i)}(y)$.
\end{proof}

\begin{defn}
Let $X \subseteq K$ be compact-open, and let $T: X \to X$ be {\maM} for $r \in \VV$.  Let $H = X/B_r(0)$ and $C: H \to \RR_{\geq 1}$ be the scaling function.  Then, we define the \emph{associated transition matrix} to be the map $A: H^2 \to \RR_{\geq 0}$ given by, for $i,j \in H$,
\[ A(i,j) = \begin{cases} 0 & i \cap T^{-1}(j) = \emptyset \\ \rho\left(1/C(i)\right) & \text{otherwise} \end{cases} \]
\end{defn}

\begin{lemma}\label{lem:maMMP}
Let $X \subseteq K$ be compact-open and $T: X \to X$ be {\maM} for $r \in \VV$; let $H = X/B_r(0)$ and let $A: H^2 \to \RR_{\geq 0}$ be the associated transition matrix.  Then:
\begin{enumerate}
\item  For $S \subseteq X$  measurable and $i \in H$ 
\[ \mu(i \cap T^{-1}(S)) = \sum_{j \in H} \mu(S \cap j) A(i,j). \]
\item $A(i,j) = \frac{1}{\rho(r)} \mu(i \cap T^{-1}(j));$
\item $A$ is a stochastic matrix on $H$;
\item $T$ is measure-preserving if and only if the sum of each column of $A$ is $1$.
\end{enumerate}
\end{lemma}
\begin{proof}
\mbox{}\\
{\noindent}{\bf (i):}\\ By disjoint additivity of $\mu$, it suffices to prove the equality in the case $S \subseteq j$ for some $j \in H$.  As the balls form a sufficient semi-ring in the Borel $\sigma$-algebra of $X$, we may in addition assume that $S$ is a ball.  Say $S = B_{r'}(a)$ for $r' \leq r$ and $a \in j$.  Then, by Corollary~\ref{cor:maMScale} we know that $i \cap T^{-1}(S)$ is either the empty set or a ball of radius $r'/C(i)$, according as whether $i \cap T^{-1}(j)$ is empty or not.  Taking measures we get
\begin{align*} 
\mu\left( i \cap T^{-1}(S) \right) &= \begin{cases} 0 & i \cap T^{-1}(j) = \emptyset \\ \rho(r'/C(i)) & \text{otherwise} \end{cases} \\ 
&= \mu(S) A(i,j) = \sum_{j \in H} \mu(S \cap j) A(i,j).
\end{align*}

\medskip
{\noindent}{\bf (ii):}\\ Put $S = j$ in (i).  Then we get
\[ \mu(i \cap T^{-1}(j)) = \mu(j \cap j) A(i,j) = \rho(r) A(i,j). \]

\medskip
{\noindent}{\bf (iii):}\\
Note that for each $i \in H$, by disjoint additivity of $\mu$ along with (ii) we have
\[ \rho(r) \sum_{j \in H} A(i,j) = \sum_{j \in H} \mu(i \cap T^{-1}(j)) = \mu(i \cap X) = \mu(i) = \rho(r). \]

\medskip
{\noindent}{\bf (iv):}\\
If $T$ is measure-preserving then for each $j \in H$ we have, by disjoint additivity of $\mu$,
\[ \sum_{i \in H} A(i,j) = \frac{1}{\rho(r)} \sum_{i \in H} \mu(i \cap T^{-1}(j)) = \frac{1}{\rho(r)} \mu(X \cap T^{-1}(j)) = \frac{1}{\rho(r)} \mu(T^{-1}(j)) =  1. \]

For the converse we use (i) and disjoint additivity:
\[ \mu(T^{-1}(S)) = \sum_{i \in H} \mu(i \cap T^{-1}(S)) = \sum_{i,j \in H} \mu(S \cap j) A(i,j) = \sum_{j \in H} \mu(S \cap j) = \mu(S). \qedhere\]
\end{proof}

\begin{prop}\label{prop:maMPhi}
Let $X \subseteq K$ be compact-open, let $T: X \to X$ be a {\maM} transformation for $r \in \VV$, with $\Sigma=(X,\mu,T)$ the corresponding measurable dynamical system.  

Let $H=X/B_r(0)$, let $A: H^2 \to \RR_{\geq 0}$ be the associated transition matrix, and $\v: H \to \RR_{\geq 0}$ the positive row vector given by $\v(i) = \rho(r)$ for $i \in H$.  Let $\Sigma' = (X_A,\mu_{A,\v},T_A)$ be the corresponding Markov shift.

Then, there exists a continuous, measure-preserving surjection $\Phi: X \to X_A$ satisfying $\Phi \circ T = T_A \circ \Phi$.  Moreover, the pre-image under $\Phi$ of a cylinder set is a ball of the same measure.
\end{prop}
\begin{proof}
For each $n \geq 0$, let $\pi_{n}: X_A \to H$ denote projection to the $n^\text{th}$ coordinate.  Let  $\phi: X \to H$ be the canonical projection.
Consider the map $\Phi: X \to X_A$ defined by  \[ \pi_{n} \circ \Phi = \phi \circ T^{n},\] i.e., the $n^\text{th}$ slot in $X_A$ denotes which element of $H$ the point $T^n(x)$ is in.  Then $\Phi \circ T = T_A \circ \Phi$ by construction.

Let $d_0,d_1,\ldots \in H$.   We will prove by induction on the number of slots specified (the ``length'' of the cylinder set $[d_0 \ldots d_\ell]$) the claim that the pre-image of the cylinder set $[d_0 \ldots d_\ell]$ is a ball of the same measure as the cylinder set.
Note that $\Phi^{-1}([d_0]) = d_0$ is a ball of the correct measure as \[ \mu_{A,\v}([d_0]) = \v(d_0) = \rho(r) = \mu(d_0) \] by construction.  Now, \[ \Phi^{-1}([d_0 \ldots d_{\ell}]) = \Phi^{-1}([d_0]) \cap T^{-1} \Phi^{-1}\left([d_1 \ldots d_\ell]\right).\]  By the inductive hypothesis, this is the intersection of two balls, and is thus again a ball; so $\Phi$ is continuous.

Noting that $\Phi^{-1}([d_1 \ldots d_\ell]) \subseteq \Phi^{-1}[d_1] = d_1$ and applying claim (i) of Lemma~\ref{lem:maMMP}, along with the inductive hypothesis, we see that this ball has the correct measure
\begin{align*} \mu\left( \Phi^{-1}([d_0 \ldots d_{\ell}]) \right) &= \sum_{j \in H} \mu\left(j \cap \Phi^{-1}([d_1 \ldots d_\ell]) \right) A(d_0,j) \\ &= A(d_0,d_1) \mu\left(\Phi^{-1}([d_1 \ldots d_\ell]) \right) \\ &= A(d_0, d_1) \mu([d_1 \ldots d_\ell]) = \mu([d_0 \ldots  d_\ell]). \end{align*}  As the cylinder sets are a sufficient semi-ring in the Borel $\sigma$-algebra of $X_A$, this shows that $\Phi$ is measure-preserving.  

Note that $\Phi$ continuous and measure-preserving implies $\Phi$ surjective: $X$ is compact and $X_A$ is Hausdorff, so the image must be closed. However, the image must have full measure and so must be dense ($\v$ positive implies that all cylinder sets, hence all open sets, have strictly positive measure).
\end{proof}

\begin{theorem}\label{thm:maMMPStruct}
Let $X, T, H, \Sigma, \Sigma', \Phi$ be as in Proposition~\ref{prop:maMPhi}.  Moreover, assume that $\Sigma$ is measure-preserving, so that $\Sigma'$ is as well.

Now, let $H = \bigsqcup_{k=1}^{n} H_k$ be a decomposition of $H$ in the sense of Proposition~\ref{prop:ergDecMarkov}, so that
\[ \Sigma' = \bigsqcup_{k=1}^{n} \Sigma'_k \] where $\Sigma'_k = (X_{A_k},\mu_{A_k,\v},T_{A_k})$.

For $k=1,\ldots,n$ define
\[ \sq{\Sigma}_k = \begin{cases} \text{restriction of $\Sigma$ to $\Phi^{-1}(X_{A_k})$} & \Card{X_{A_k}} < \infty \\ \Sigma'_k & \text{otherwise} \end{cases} \]
then we have an isomorphism of topological and measurable dynamical systems
\[ \Sigma \isom \bigsqcup_{k=1}^{n} \sq{\Sigma}_k. \]
Moreover, each term in this decomposition is either locally an isometry or ergodic Markov, according as whether $\Card{X_{A_k}} < \infty$ or not.
\end{theorem}
\begin{proof}  By Lemma~\ref{lem:maMMP}, $\Sigma'$ is a quotient of $\Sigma$, so $\Sigma$ measure-preseving implies $\Sigma'$ measure-preserving.

For $k=1,\ldots,n$, denote $C_k = \Phi^{-1}(X_{A_k})$, $\mu_k = \left.\mu\right|_{C_k}$, $T_k=\left.T\right|_{C_k}$.  The decomposition of $\Sigma'$ induces the following decomposition of $\Sigma$:
\[ \Sigma = \bigsqcup_{k=1}^{n} (C_k,\mu_k,T_k). \]  To complete the proof of the proposition, it suffices to show that $(C_k,\mu_k,T_k) \isom \sq{\Sigma}_k$ for $k=1,\ldots,n$ as topological and measurable dynamical systems, and to classify them as being locally isometries and ergodic Markov in the two cases.  We now handle the two cases separately:  

{\noindent}{\it Case 1: $\Card{X_{A_k}}<\infty$ }\\
If $\Card{X_{A_k}} < \infty$, then the isomorphism $(C_k,\mu_k,T_k) \isom \sq{\Sigma}_k$ follows by definition.  Note that the measure on $\Sigma_k'$ is necessarily atomic; as it is ergodic, it must in fact be the inverse orbit of a single atom.  As $T_{A}$, hence $T_{A_k}$, is measure-preserving, each of the atoms must have equal measure.  It follows that each element $x \in X_{A_k}$ is of the form \[ x = (d_0,d_1,\ldots,d_\ell,d_0,\ldots,d_\ell,d_0,\ldots,d_\ell,\ldots), \] with $A(d_0,d_1)=A(d_1,d_2)=\ldots=A(d_\ell,d_0)=1$.  Then, $\Phi^{-1}(x) = d_0$, where $C(d_0)=1$ (here, $C$ is that from the definition of {\maM}).  

So, $C_k$ must be a collection of $r$-balls with $C(x)=1$ for $x \in C_k$.  Then, for $x,y \in C_k$ with $|x-y| \leq r$ we have $|T(x)-T(y)|=C(x) |x-y| = |x-y|$.  This shows that $\sq{\Sigma}_k$ is locally an isometry, as desired.

{\noindent}{\it Case 2: $\Card{X_{A_k}}=\infty$ }\\
If $\Card{X_{A_k}} = \infty$, then we claim that $\Phi$ induces an isomorphism $(C_k,\mu_k,T_k) \isom \sq{\Sigma}_k$.  In a measure-preserving Markov shift, any atoms must have finite inverse orbit; so $\Sigma'_k$ ergodic and $\Card{X_{A_k}} = \infty$ implies that $\mu_{A_k,\v_k}$ is non-atomic. Recall that $\Phi$ is surjective.  We claim that it is also injective.  For $x \in X$ let $d_n = \pi_H T^n(x)$ for $n=0,1,\ldots$.  Then, \[ \Phi^{-1}(\Phi(x)) = \bigcap_{\ell \geq 0} \Phi^{-1}\left([d_0 \ldots d_\ell] \right). \]  We have from Proposition~\ref{prop:maMPhi} that each of these pre-images is a ball.  Then, $\Phi^{-1}(\Phi(x))$ is the intersection of a nested family of balls.  If the intersection contains more than a single point, then the radii of the balls do not go to $0$, and so the intersection has non-empty interior and thus positive measure.  Now, the measure on $X_{A_k}$ is non-atomic, so $\mu_{A_k,\v_k}(\Phi(x))=0$.  As $\Phi$ is measure-preserving, this implies that $\mu(\Phi^{-1}(\Phi(x)))=0$; by the above considerations this implies that $\Phi^{-1}(\Phi(x))$ contains at most one point.  So, $\Phi$ is injective.

Then, $\Phi$ is a continuous, measure-preserving  bijection.  Observe that $\Phi$ takes closed sets to closed sets by compactness, so its inverse is also continuous. This also implies that $\Phi^{-1}$ is measurable, and then $\Phi$ measure-preserving implies $\Phi^{-1}$ measure-preserving.  So, $\Phi$ is an isomorphism of topological and measurable dynamic systems $(C_k,\mu_k,T_k) \isom \sq{\Sigma}_k$ as desired.  As the later is ergodic Markov, the former is as well.
\end{proof}

\begin{corollary}
Let $X \subseteq K$ be compact-open and $T: X \to X$ a measure-preserving {\maM} transformation.  If $T$ is ergodic then it is either Markov or locally an isometry.  In particular, if it is weakly mixing then it also Markov and so mixing.  So, for a  measure-preserving {\maM} transformation on a compact-open $X$, weakly mixing implies mixing.
\end{corollary}
\begin{proof}
If $T$ is ergodic, then the decomposition in Theorem~\ref{thm:maMMPStruct} must be trivial.  So, $T$ must be either Markov or locally an isometry.  If it is locally an isometry, then it cannot be weakly mixing.  So, weakly mixing implies weakly mixing Markov which in turn implies mixing.
\end{proof}

\begin{corollary}
For a {\maM} transformation, the following properties depend only on the associated transition matrix:
\begin{enumerate}
\item Measure-preserving;
\item Weakly mixing, mixing, exact, Bernoulli.
\end{enumerate}
\end{corollary}
\begin{proof}
By Lemma~\ref{lem:maMMP}, the property of being measure-preserving depends only on the associated transition matrix.

Note that the decomposition in Theorem~\ref{thm:maMMPStruct} depends only on the associated transition matrix.  Given an associated transition matrix, we have the following cases:
\begin{enumerate}
\item The decomposition is trivial, and the system is a local isometry.  Then, it is not weakly mixing (or any of the stronger properties listed).
\item The decomposition is trivial, and the system is ergodic Markov.  In this case, the system is determined up to isomorphism by the matrix.
\item The decomposition is not trivial.  In this case, the system is not ergodic and cannot satisfy any of the stronger properties listed. \qedhere
\end{enumerate}
\end{proof}

\section{Polynomial approximation in $\OO$}\label{sec:polApprox}
The above results dealt with $\C^1$ functions, extending to polynomial maps as a special case.  In the next sections we will be interested in finding polynomial maps with specified associated transition matrices.  In preparation for this, we will need some results on the approximation of continuous maps $\OO \to K$.  For the reader's convenience, we will sketch here the definitions and results of \cite{Amice}, slightly simplified for our applications.

Say $X \subseteq \OO$ is compact-open.  Moreover, assume that $X$ is a finite union of $r$-balls for $r\in \VV$.  Then, for  $r' \leq r$ each $r'$-ball contained in $X$ is a union of precisely $q$ balls of radius $|\pi| r'$ contained in $X$.  In the terminology of \cite{Amice}, this makes $X$ a \emph{regular valued compact} (\emph{compact valu\'e r\'egulier} in the original French).

For $k \geq \log_{|\pi|}{r}$, we may define $H_k = X/B_{|\pi|^{k}}(0)$, and a projection map $\pi_k: X \to H_k$.  Then, we say that a sequence $\{ u_k \in X : k \in \NN \}$ is \emph{very well distributed} (\emph{tr\`es bien r\'epartie}) if for each $k \geq \log_{|\pi|} r$, $h \in H_k$, and $m \geq 1$ we have
\[ \Card{\{ i < m \Card{H_k} : u_i \in h \}} = m. \]
That is, the terms of the sequence must be equally distributed among the possible values $\mod{\pp^k}$ for $k \geq \log_{|\pi|} r$.  Note that the condition that the $\{u_k\}$ are very well distributed implies that they are distinct.

Now, given such a sequence $\{u_0,u_1,\ldots\}$, we may define the \emph{{\cip}} for $k \geq 0$:
\[ P_k(x) = (x-u_0)(x-u_1)\cdots(x-u_{k-1})\qquad\text{and}\qquad Q_k(x) = \frac{P_k(x)}{P_k(u_k)}. \]

Then, we may summarize some of the results of \cite[\S II.6.2]{Amice} as follows:
\begin{theorem}[Amice]\label{thm:polApprox1}
Let $X \subseteq \OO$ be compact-open, and let $\{ u_k \}$ be a very well distributed sequence with values in $X$ with $P_k, Q_k$ the {\cip}.   Let $f: X \to K$ be continuous,  and for $k \geq 0$ set
\[ a_k = P_k(u_k) \left( \sum_{j=0}^{k} \frac{f(u_j)}{P'_{k+1}(u_j)} \right). \]

Then:
\begin{enumerate}
\item $|a_k| \to 0$ as $k \to \infty$;
\item $ \sum_{k \geq 0} a_k Q_k(x) \to f(x) $ uniformly on $X$;
\item The $a_k$ are determined by (ii);
\item $ \sup_{x \in X} |f(x)| = \sup_{k \in \NN} |a_k|$.
\end{enumerate}
\end{theorem}

A very well distributed sequence $\{u_k\}$ is said to be \emph{well ordered} (\emph{bien ordonn\'ee}) if $|u_n-u_m|=|\pi|^{v_q(n-m)}$ for all $n,m \geq 0$ where $v_q(n-m)$ is the exact power of $q$ dividing $n-m \in \ZZ$.  Following our sources, we will call such a sequence \emph{\TBRBO} (\emph{tr\`es bien r\'epartie bien ordonn\'ee}).
This allows us to state results of Helsmoortel and Barsky, characterizing Lipschitz and $\C^1$ functions on $\OO$ in terms of the coefficients in their expansions.  This result may be found in \cite{Barsky}.
\begin{theorem}[Helsmoortel, Barsky]\label{thm:polApprox2}
Let $\{ u_k \}$ be a {\TBRBO} sequence with values in $\OO$, with $P_k, Q_k$ the {\cip}.  Let $f:\OO \to K$ be continuous with 
\[ f(x) = \sum_{k \geq 0} a_k Q_k(x) \] the expansion of $f$ in the sense of Theorem~\ref{thm:polApprox1}.  For $k \geq 1$, define \[ \kappa_k = |\pi|^{-\lfloor \log_{q} k \rfloor}. \] Then:
\begin{enumerate}
\item $f$ is $r$-Lipschitz if and only if $r \leq \kappa_k |a_k|$ for all $k \geq 1$;
\item $f \in \C^1(\OO)$ if and only if $\kappa_k |a_k| \to 0$ as $k \to \infty$.
\end{enumerate}
\end{theorem}

\begin{example}\label{ex:TBRBO}
Note that $\{0, 1, 2, \ldots \} \subseteq \ZZ_p$ satisfies the conditions for being a very well distributed sequence, and is in fact trivially {\TBRBO}  Then,
\[ Q_k(x) = \frac{x(x-1)\cdots(x-k+1)}{k\cdot (k-1) \cdots \cdot 1} = {x \choose k}. \]  So, in this case the above reduces to the Mahler expansion.

More generally:
Let $a_0,\ldots,a_{q-1}$ be a complete set of coset representatives for $\OO/\pp$.  For $k \in \NN$, we will define $u_k$ in terms of the base-$q$ expansion of $k$:
\[ k = \sum_{i=0}^{\ell} k_i q^i \longmapsto \sum_{i=0}^{\ell} a_{k_i} \pi^i = u_k. \] Then, say we have $n,m \in \NN$ with $n = \sum_{i \geq 0} n_i q^i$ and $m = \sum_{i \geq 0} m_i q^i$.  Let $\ell = v_q(i-j) = \min\{ i : n_i \neq m_i \}$.  Then, \[ |u_n - u_m| = |\pi|^\ell = |\pi|^{v_q(n-m)}, \] and $\{u_k\}$ is {\TBRBO}  In particular, this implies that there is always a {\TBRBO} sequence for $\OO$, and {\cip} such that the results cited in this section hold.
\end{example}

Now, we establish a lemma that will be of particular interest to us:
\begin{lemma}\label{lem:binomBounds}
Let $\{u_n\}$ be a {\TBRBO} sequence in $\OO$ with {\cip} $P_k, Q_k$   Then, for $k \in \NN$:
\begin{enumerate}
\item $Q_k(\OO) \subseteq \OO$;
\item $Q_k$ is $\kappa_k$-Lipschitz, with $\kappa_k$ as in Theorem~\ref{thm:polApprox2};
\item If $k = q^\ell$ for some $\ell \geq 0$, then \[ |Q_k(x)-Q_k(y)| = \kappa_k |x-y| \text{ for all $x,y \in \OO$ with $|x-y| \leq 1/\kappa_k$}. \]
\end{enumerate}
\end{lemma}
\begin{proof}
Claim (i) follows by applying Theorem~\ref{thm:polApprox1}(iv) with $f = Q_k$ (so that $a_i = 1$ for $i=k$ and $0$ otherwise).  Claim (ii) follows similarly from Theorem~\ref{thm:polApprox2}.

Assume $k = q^\ell$.  Fix $x, y \in \OO$ with $|x-y| \leq 1/\kappa_k$; as the $\{u_n\}$ are very well distributed there is some $m \in \{0,\ldots,k-1\}$ such that $|x-u_m| \leq 1/\kappa_k$ (hence also $|y-u_m| \leq 1/\kappa_k$).

Define a polynomial \[ S_j(z) = \sum_{0 \leq i_1 < i_2 < \cdots < i_j < k} \frac{(y-u_0) \cdots \widehat{(y-u_{i_1})} \cdots \widehat{(y-u_{i_j})} \cdots (y-u_{k-1})}{(u_k-u_{k-1})\cdots(u_k-u_0)}. \]  Then, we may observe that 
\[ Q_k(x') - Q_k(y') = \frac{(x'-u_{k-1})\cdots(x'-u_0) - (y'-u_{k-1})\cdots(y'-u_0)}{(u_k-u_{k-1})\cdots(u_k-u_0)}   = \sum_{j=1}^k (x'-y')^j S_j(y') \] for any $x',y' \in \OO$.  Then,
\begin{align*}
Q_k(x) - Q_k(y) &= \left(Q_k(x) - Q_k(u_m)\right) - \left(Q_k(y)-Q_k(u_m)\right) \\
&= (x-y) S_1(u_m) + \sum_{j=2}^{k} S_j(u_m)\left( (x-u_m)^j - (y-u_m)^j \right).\end{align*}

We will prove the following two statements, which together with the strong triangle inequality and the  previous expression imply our desired result:
\begin{enumerate}
\item $\left|S_1(u_m)\right| = \kappa_k$;
\item $\left|S_j(u_m)\right| < \kappa_k^j$ for $1 < j \leq k$.
\end{enumerate}
That this suffices is clear, for the $j=1$ term will dominate in valuation.

Observe that
\[ S_1(u_m) = \frac{(y-u_0) \cdots \oh{(y-u_m)} \cdots (y-u_{k-1})}{(u_k-u_0) \cdots (u_k-u_{k-1})} \] and that
\begin{align*} S_j(u_m) &= \sum_{{0 \leq i_1 < i_2 < \cdots < i_j < k}\atop{m \in \{i_1,\ldots,i_j\}}}  \frac{(u_m-u_0) \cdots \widehat{(u_m-u_{i_1})} \cdots \widehat{(u_m-u_{i_j})} \cdots (u_m-u_{k-1})}{(u_k-u_{k-1})\cdots(u_k-u_0)}.
\end{align*}

Suppose $\{v_n\}$ is a very well distributed sequence.  Then, it is easy to check that $\{v_0,\ldots,v_{k-1}\}$ must contain precisely $q^{\ell'}$ elements bounded by $\pi^{\ell'}$ for each $\ell' \leq \ell$.  Now, observe that both $\{u_m-u_n: n \in \NN\}$ and $\{u_k-u_n : n \in \NN\}$ are very well distributed. Let $m'$ be the unique index in $\{0,\ldots,k-1\}$ such that $|u_k-u_{m'}| \leq 1/\kappa_k$; the very well distributed property of $\{u_n\}$ implies that this is in fact an equality.  Our previous count implies that we must have \[ \left| \frac{(y-u_0) \cdots \oh{(y-u_m)} \cdots (y-u_{k-1})}{(u_k-u_0) \cdots \oh{(u_k-u_{m'})} \cdots (u_k - u_{k-1})} \right| = 1. \]  So, \[ \left| S_1(u_m) \right| = \left| \frac{1}{u_k-u_{m'}} \right| = \kappa_k. \]  

Now,
\[ \left| \frac{S_j(u_m)}{S_i(u_m)} \right| = \left|\sum_{{0 \leq i_1 < i_2 < \cdots < i_{j-1} < k}\atop{m \notin \{i_1,\ldots,i_{j_1}\}}} \frac{1}{(u_m-u_{i_1})\cdots(u_m-u_{i_{j-1}})}\right| < \kappa_k^{j-1}, \] for $|u_m-u_{i_1}|,\ldots,|u_m-u_{i_{j-1}}| > 1/\kappa_k$ as $\{u_n\}$ is very well distributed (and so the first $k$ elements must be in disjoint $1/\kappa_k$-balls).  This completes our proof.
\end{proof}


\section{Polynomial maps on $\OO$ realizing {\maM} transformations}\label{sec:polMaM}
Sections~\ref{sec:mpC1} and \ref{sec:structMaM} characterize measure-preserving polynomial transformations on $\OO$ in terms of {\maM} transformations.  However, we have shown the existence of only a handful of such maps.  In this section, we will show that in fact the polynomials, in a sense, provide a representative class among the measure-preserving {\maM} maps.

We begin with a lemma giving sufficient conditions for two maps to have the same associated transition matrices.
\begin{lemma}\label{lem:matrixApprox}
Suppose $T: \OO \to \OO$ is {\maM} for $r \in \VV$; set $H = \OO/B_r(0)$, let $C: H \to \RR_{\geq 1}$ be the scaling function for $T$, and let $A:H^2 \to \RR_{\geq 0}$ be the associated transition matrix for $T$.  Suppose in addition that $S: \OO \to \OO$ is a transformation such that the difference $R = T - S$ satisfies
\begin{enumerate}
\item $|R(x) - R(y)| < C(x) |x-y|$ whenever $0 < |x-y| \leq r$;
\item $|R(x)| \leq r C(x)$ for all $x$.
\end{enumerate}
Then, $S$ is {\maM} for $r \in \VV$, with scaling function $C$ and associated transition matrix $A$.
\end{lemma}
\begin{proof}
For $0 < |x-y| \leq r$ we have
\[ |S(x)-S(y)|=|T(x)-T(y) + R(x) - R(y)|=|T(x)-T(y)| = C(x) |x-y| \]  by the strong triangle inequality.  Indeed, $|T(x)-T(y)| = C(x) |x-y|$, which by (i) is strictly greater than $|R(x)-R(y)|$.  So, $S$ is {\maM} for $r \in \VV$, with scaling function $C$.

Now, it remains to verify that $ i \cap T^{-1}(j) = \emptyset \Leftrightarrow i \cap S^{-1}(j) = \emptyset $ for $i,j \in H$.  For this, it suffices to show that $T(B_r(x))=S(B_r(x))$ for all $x \in \OO$.  Indeed, applying Lemma~\ref{lem:maMScale} and (ii) yields
\[ T(B_r(x)) = B_{r C(x)}(T(x)) = B_{r C(x)}(S(x)) = S(B_r(x)). \qedhere \]
\end{proof}

For $S \subseteq K$ we say that $T: S \to K$ is \emph{affine} if it is given by $x \mapsto a x + b$ for some constants $a,b \in K$.  We say that $T: \OO \to K$ is \emph{locally affine} if for each $x \in \OO$ there exists a $r \in \VV$ such that $\left. T \right|_{B_r(x)}$ is affine.

%
%
%

\begin{theorem}\label{thm:polyRepMaM}
Let $r \in \VV$ and $H=\OO/B_r(0)$.  Let $A$ be a stochastic matrix on $H$. Then, let \[ \mathcal{T}_A = \{ T\text{ {\maM} for $r$} : A\text{ is the associated transition matrix for $T$} \}. \]  If $\mathcal{T}_A$ is non-empty then:
\begin{enumerate}
\item $\mathcal{T}_A$ contains a locally affine transformation;
\item $\mathcal{T}_A$ contains infinitely many polynomials.
\end{enumerate}
\end{theorem}
\begin{proof}\mbox{}\\
{\noindent}{\bf(i): }\\
Suppose $T \in \mathcal{T}_A$.  Let $C: H \to \RR_{\geq 1}$ be the scaling function in the definition of {\maM} and observe that its image is contained in $\VV$.  Let $S: H \to K$ be any function satisfying $|S(h)| = C(h)$ for all $h \in H$.  Then, for each $h \in H$ we observe that $T(h) = \{ T(x): x \in h \}$ and $S(h) h = \{ S(h) x : x \in h \}$ are both balls of radius $C(h) r$.  So, there exists a function (indeed, many of functions) $M: H \to K$ such that $T(h) = S(h) h + M(h)$.  Regarding $S$ and $M$ as functions with domain $\OO$ via the quotient map $\OO \to H$, we may define a locally affine transformation $T_{S,M}: \OO\to \OO$ by the formula $T_{S,M}(x) = S(x) x + M(x)$.  

Since $S, M$ are constant on elements of $H$ and $|S(h)| = C(h)$, it follows at once that $T_{S,M}$ is {\maM} for $r$ with scaling function $C$.  By construction, $T_{S,M}(h) = T(h)$ (which is also why $T_{S,M} \OO \subset\OO$), so that $T_{S,M} \in \mathcal{T}_A$ by the definition of the associated transition matrix.

\medskip
{\noindent}{\bf (ii): }\\
Let $T \in \mathcal{T}_A$.  By (i), we may assume that $T$ is locally affine and hence strictly differentiable.  Let $\{u_k\}$ be a {\TBRBO} sequence in $\OO$ (which must exist by Example~\ref{ex:TBRBO}).  Let
\[ T = \sum_{k \geq 0} a_k Q_k. \]  
be the decomposition of $T$ in the sense of Theorem~\ref{thm:polApprox1}

Take $\alpha \in \VV$ with $\alpha < 1$.  By Theorem~\ref{thm:polApprox2}, $\kappa_k |a_k| \to 0$ as $k \to \infty$, so there exists an $N \in \NN$ such that for $k > N$ we have $\kappa_k |a_k| \leq \alpha < 1$.  Moreover, take $N$ such that $N > \Card{H}$.

Let \[ f(x) = \sum_{k=0}^{N} a_k Q_k(x).\]  Note that $f$ is a polynomial.  Set \[ R(x) = T(x)-f(x) = \sum_{k > N} a_k Q_k(x). \]  Our bound on $\kappa_k |a_k|$ along with the choice $N > \Card{H} = q^{\log_{|\pi|} r}$ implies that $|a_k| < 1/\kappa_N \leq  r$ for $k > N$; so Theorem~\ref{thm:polApprox1}(iv) implies that $|R(x)| \leq r$ for all $x \in \OO$.  Lemma~\ref{lem:binomBounds} implies that $R$ is $\alpha$-Lipschitz, so that $|R(x) - R(y)| \leq \alpha |x-y|$ for $x,y \in \OO$.  Noting that $\alpha < 1 \leq C(x)$ for all $x \in \OO$, we observe that we may apply Lemma~\ref{lem:matrixApprox} to conclude that $f \in \mathcal{T}_A$.

Note that if 
\[ g(x) = \sum_{k > N} b_k Q_k(x) \] is a polynomial such that $\kappa_k |b_k| \leq \alpha$, then the above argument also shows that $f + g \in \mathcal{T}_A$.  So, there are indeed infinitely many polynomials in $\mathcal{T}_A$.
\end{proof}

In particular, Theorem~\ref{thm:polyRepMaM} shows the existence of measure-preserving mixing transformations on the $p$-adics given by polynomial maps.  We can also use this method to compute explicit examples of such maps, but it is not particularly enlightening to do so.  

\section{Polynomial Bernoulli maps on $\OO$}\label{sec:polBern}
The construction of the preceding section gives infinite classes of measure-preserving polynomials with different kinds of measurable dynamics.  Among these maps are Markov mixing maps.  We will now study the class of such polynomials whose associated transition matrix has all entries equal, in which case the Markov transformation is in fact Bernoulli.  The main upshot of this study is a class of explicitly given and relatively simple measure-preserving Bernoulli polynomial maps.

\begin{defn}
We say that a measure-preserving {\maM} map $T: \OO \to \OO$ is \emph{isometrically Bernoulli} for $r \in \VV$ if it is {\maM} for $r \in \VV$ and all entries of the associated transition matrix are equal.
\end{defn}

Let $V = \OO/\pp^\ell \isom \FF_q^\ell$ and define
\[ B_V = \left(\prod_{i \geq 0} V, \mu_V, T_V \right) \] where $\mu_V$ is the product probability measure, and $T_V$ the left-shift.  We may let $d'_V$ be the quotient metric on $V$. Then, we may define a metric $d_V$ on $\prod_{i \geq 0} V$ by
\[ d_V\left( (a_0,a_1,a_2,\ldots), (b_0,b_1,b_2,\ldots) \right) = |\pi|^{-\ell (m-1)} d'_V(a_{m},b_{m}) \text{ where $m=\min\{i: a_i \neq b_i\}$}. \]
We give two justifications for this metric:
\begin{enumerate}
\item View elements of $V$ as $\ell$-tuples under the isomorphism $\FF_q^\ell \isom V$ corresponding to $\pi$-adic expansion (i.e., the isomorphism induced by the map shown in (ii)).  Then, expanding each element in the product to a $\ell$-tuple, $d_V$ is just the dictionary metric (with base $|\pi|$).
\item For each $a \in V$ we may let $\ol{a} \in \OO$ be a coset representative for the quotient.  Then, the map
\[ (a_0,a_1,\ldots) \longmapsto \sum_{i \geq 0} \ol{a_i} \pi^{\ell i} \]
gives a bijection $\prod_{i \geq 0} V \to \OO$.  This metric is the unique metric making this map an isometry.
\end{enumerate}

Now, the term isometrically Bernoulli is partially motivated by the following:
\begin{lemma}\label{lem:isoBernCat}
Let $T: \OO \to \OO$ be a transformation and let $\ell \geq 1$.  Then, the following are equivalent:
\begin{enumerate}
\item $T$ is isometrically Bernoulli for $r = |\pi|^\ell$;
\item For all $x,y \in \OO$ satisfying $|x-y| \leq |\pi|^\ell$, \[ |T(x) - T(y)| = |\pi|^{-\ell} |x-y|. \]
\item Let $V =\OO/\pp^\ell$. There exists an invertible isometry $\Phi: \OO \to \prod_{i \geq 0} V$ such that $\Phi \circ T = T_V \circ \Phi$; that is, $(\OO,\mu,T)$ is metrically isomorphic to $B_V$.
\end{enumerate}
\end{lemma}
\begin{proof}
{\noindent}{\bf (i)$\Rightarrow$(ii): }\\
Let $H = \OO/B_r(0)$, and $A: H^2 \to \RR_{\geq 0}$ the associated transition matrix.  Note that if $T$ is isometrically Bernoulli, then each entry of $A$ must be equal, and hence must be equal to $\tfrac{1}{\Card{H}}=\rho(r)$.  Now, $T$ must be {\maM} for $r \in \VV$, so $|T(x)-T(y)| = C(x) |x-y|$ for $|x-y| \leq r = |\pi|^\ell$.  But, we must have $\rho(1/C(x))=\rho(r)$, so $C(x)=1/r = |\pi|^{-\ell}$.  

\medskip
{\noindent}{\bf (ii)$\Rightarrow$(iii): }\\
Let $V = \OO/\pp^\ell$.  Now, (ii) implies that $T$ is {\maM} for $r$.  Letting $A$ be the associated transition matrix, we readily note that all non-zero entries of $A$ must be equal to $\rho(|\pi|^\ell)$; as $A$ is a stochastic matrix, this implies that all entries of $A$ are non-zero.

Now, let $\Sigma' = (X_A,\mu_{A,\v},T_{A})$ be as in Theorem~\ref{thm:maMMPStruct}.  We see that $B_V = \Sigma'$.  We observed above that all entries of $A$ are non-zero; then, $A$ is irreducible and Theorem~\ref{thm:maMMPStruct} gives us a topological and measurable isomorphism $\Phi: \OO \to X_A$.  Note that the balls of $X_A$ with respect to $d_V$ are just the cylinder sets.  Moreover, one may check that for each $m \geq 0$, $X_A$ is a disjoint union of $q^m$ balls of radius $r = |\pi|^m$, which must then each have measure $q^{-m} = \rho(r)$.  Then, Proposition~\ref{prop:maMPhi} implies that $\Phi^{-1}$ takes balls of a given radius to balls of the same radius; moreover, $\Phi^{-1}$ must take each of the $q^m$ distinct balls of radius $|\pi|^m$ in $X_A$ to a distinct ball of radius $|\pi|^m$ in $\OO$.  So each ball of radius $|\pi|^m$ in $\OO$ must be the pre-image of precisely one ball of the same radius in $X_A$.  It follows that $\Phi$ and $\Phi^{-1}$ are both isometries.

\medskip
{\noindent}{\bf (iii)$\Rightarrow$(i): }\\
Note that for $x,y \in \OO$ we have
\[ |T(x)-T(y)| = d_V(\Phi(T(x)),\Phi(T(y))) = d_V(T_V(\Phi(x)) , T_V(\Phi(y))). \]  Then, for $d_V(\Phi(x),\Phi(y)) = |x-y| \leq |\pi|^\ell$ we compute
\[ d_V(T_V(\Phi(x)) , T_V(\Phi(y))) = |\pi|^{-\ell} d_V(\Phi(x),\Phi(y)) = |\pi|^{-\ell} |x-y|. \qedhere\]
\end{proof}

\begin{remark}
Note that item (ii) of Lemma~\ref{lem:isoBernCat} implies that $T$ is $|\pi|^{-\ell}$-Lipschitz.
\end{remark}

Now, we may combine Lemma~\ref{lem:binomBounds} with Lemma~\ref{lem:isoBernCat} to get:
\begin{corollary}\label{cor:isoBern}
Let $\{u_k\}$ be a {\TBRBO} sequence with values in $\OO$, with {\cip} $P_k, Q_k$.  Say $T: \OO \to \OO$ is given by the expansion, in the sense of Theorem~\ref{thm:polApprox1},
\[ T(x) = \sum_{k \geq 0} a_k Q_k(x)\text{, with $a_k \in \OO$, $|a_k| \to 0$}.\]  Assume that 
\begin{enumerate}
\item $M = \max_{k \geq 0} \kappa_k |a_k|$ exists, where $\kappa_k$ is as in Theorem~\ref{thm:polApprox2};
\item There is a unique $k_M \geq 0$ attaining this maximum, and moreover it is of the form $k_M=q^\ell$ for some $\ell \geq 1$;
\item $|a_{k_M}| = 1$ (hence, $M = \kappa_{k_M}$).
\end{enumerate}

Then, $T$ is isometrically Bernoulli for $r=1/M \in \VV$.
\end{corollary}
\begin{proof}
As $k_M$ is the unique value attaining the maximum, the strong triangle inequality and Lemma~\ref{lem:binomBounds} imply that 
\[ |T(x)-T(y)| = \left| \sum_{k \geq 0} a_k \left[Q_k(x)-Q_k(y)\right] \right| = \kappa_{k_M} |a_{k_M}| |x-y| = M |x-y| \]
for $|x-y| \leq 1/\kappa_{k_M} = 1/M$.  Then, our claim follows by Lemma~\ref{lem:isoBernCat}.
\end{proof}

\begin{example}\label{ex:mahlerCondition}
Let $K=\QQ_p$ and $\OO = \ZZ_p$.  Then, $\{ 0, 1, 2, \ldots\}$ is a {\TBRBO} sequence, and letting $P_k, Q_k$ be the {\cip} we can check that $Q_k(x)={x\choose k}$ (cf. Example~\ref{ex:TBRBO}). In this case, $\kappa_k=p^{\lfloor\log_p k\rfloor}$. Therefore, we can rewrite the sufficient conditions in Corollary~\ref{cor:isoBern} as  follows.  Given $T:\ZZ_p\longrightarrow\ZZ_p$ defined by 
$$
T(x)=\sum_{k\geq0}a_k {x\choose k}
$$
assume that 
\begin{enumerate}
\item $M=\max_{k\geq0}|a_k|p^{\lfloor\log_p k\rfloor}$ exists;
\item There is a unique $k_M\geq0$ attaining this maximum, and moreover it is of the form $k_M=p^\ell$ for some $\ell\geq 1$;
\item $|a _{k_M}|=1$ (thus, $M=p^\ell$).
\end{enumerate}
Then $T$ is isometrically Bernoulli for $r=p^{-\ell}$. In particular, the polynomials ${x\choose p^\ell}$ for $\ell>0$ clearly satisfy these conditions, and so each defines a Bernoulli transformation on $\ZZ_p$.

Note, in particular, that this criterion applies to any map $\ZZ_p \to \ZZ_p$ defined by $u {x \choose p} + F(x)$ with $u \in \ZZ_p^\times$ and $F \in \ZZ_p[x]$.  An example of such a map is that given by $\frac{x^p-x}{p}$ from \cite{WoodcockSmart}.
\end{example}

In this context, the polynomials ${x \choose p}$ and $\frac{x^p-x}{p}$ are in a sense the most natural isometrically Bernoulli maps:
\begin{example}\label{ex:x^p-x}
Take a set of coset representatives for $\ZZ_p/p \ZZ_p$.  Then, using Example~\ref{ex:TBRBO} we may form a {\TBRBO} sequence, and then the $p^\text{th}$ corresponding interpolating polynomial (and unit multiples of it) will be Bernoulli by Corollary~\ref{cor:isoBern}.

Let's look at the two most common sets of coset representatives for the quotient $\ZZ_p / p\ZZ_p$:
\begin{enumerate}
\item Take as coset representatives $0,1,2,\ldots,p-1$.  The resulting {\TBRBO} sequence is $\{0,1,\ldots\}$.  Then, $P_p(x) = x(x-1)\cdots(x-p+1)$ and $Q_p(x) = {x \choose p}$ is the $p^\text{th}$ corresponding interpolating polynomial.
\item Take as coset representatives $0$ and the $(p-1)^\text{st}$ roots of unity (there are exactly $p-1$ by Hensel's Lemma); these are called the ``Teichm\"uller representatives.''  Then, $P_p(x) = x^p-x$ and
\[ Q_p(x) = \frac{P_p(x)}{P_p(p)} = \frac{1}{p^{p-1}-1} \frac{x^p-x}{p}. \]
\end{enumerate}

So, the polynomials ${x \choose p}$ and $\frac{x^p-x}{p}$ (up to unit) are analogs, arising by the same construction from the two most common choices for the coset representatives of $\ZZ_p / p\ZZ_p$.
\end{example}

\begin{example}
Let $K = \FF_q((t))$ and $\OO=\FF_q[[t]]$.  We may construct a {\TBRBO} sequence as in Example~\ref{ex:TBRBO}, having  $0,1,2,\ldots,q-1,t$ as its first $q+1$ terms.  Then $$Q_q(x)=\frac{x(x-1)\ldots(x-q+1)}{t(t-1)(t-2)\ldots(t-q+1)},$$ which is isometrically Bernoulli by Corollary~\ref{cor:isoBern}. However, $t-1,$ $t-2,\ldots,t-q+1$ are all units in $\OO$, thus $$t(t-1)\ldots(t-q+1)Q_q(x)=\frac{x(x-1)\ldots(x-q+1)}{t}$$ defines a Bernoulli transformation as well.
\end{example}


Now, we will give two examples of isometrically Bernoulli polynomial maps on the rings of integers of finite extensions of $\QQ_p$.  First, we briefly motivate our choice of examples.  For $K$ a finite extension of $\QQ_p$, let $n=[K:\QQ_p]$,$f=[\OO/\pp:\FF_p]$, and $e = \log_{|\pi|} |p|$.  It is a standard result that $ef = n$.  It is evident that the nature of how $\OO$ compares to $\ZZ_p$ depends on the values of $e$ and $f$.  The two extreme cases are $f=1, e=n$ (in which case we say that the extension is \emph{totally ramified}) and $e=1, f=n$ (in which case we say that the extension is \emph{unramified}).  We give an example from each of these two extremes.  For more background on the relevant theory, including the ``standard'' results invoked in this paragraph and in the following two examples see \cite{Serre}, particularly Ch. I \S7, 8., Ch. III \S 5, Ch. IV \S 4. 

\begin{example}
Take $p>2$ and let
\[ K = \QQ_p(\zeta_p) \text{ where $\zeta_p$ is a primitive $p^\text{th}$ root of unity.} \]
%

It is a standard result that $1-\zeta_p$ may be taken as a uniformizing parameter and that the extension is totally ramified and so the set $\{0,\ldots,p-1\}$ gives a complete set of coset representatives for $\OO/\pp$.  We may construct a {\TBRBO} sequence as in Example~\ref{ex:TBRBO}.  The first $p$ terms would be just $0, \ldots,p-1$, with the next term $1-\zeta_p$.  The first $p^2$ terms would be $\{i+j(1-\zeta_p)\}$ for $0 \leq i,j < p$, with the next term $(1-\zeta_p)^2$.  Noting that $q = p$ and applying Corollary~\ref{cor:isoBern} shows that the transformations defined by the polynomials
\[ \frac{x(x-1)\ldots(x-p+1)}{1-\zeta_p} \] and
\[ \frac{1}{(1-\zeta_p)^3} \prod_{0 \leq i,j < p} \left(x-i - j(1-\zeta_p)\right) \] are isometrically Bernoulli for $r = |\pi|^{1} = |p|^{\tfrac{1}{p-1}}$ and $r=|\pi|^{2} = |p|^{\tfrac{2}{p-1}}$, respectively.
\end{example}

\begin{example}
Take $f > 1$ and let
\[ K = \QQ_p(\zeta) \text{ where $\zeta$ is a primitive $(p^f-1)^\text{th}$ root of unity.} \]
It is a standard result that $K$ is the unique unramified extension of $\QQ_p$ of degree $f$.  So, $p$ may be taken as a uniformizing parameter.  Let $\ol{\zeta} \in \OO/\pp$ be the image of $\zeta$ under the quotient map.  We note that $\ol{\zeta}$ must generate the residue field extension, i.e., $\OO/\pp = \FF_p(\ol{\zeta}) = \FF_p[\ol{\zeta}]$.  So, $S = \{ a_0 + a_1 \zeta + \ldots + a_{f-1} \zeta^{f-1} \}$, with $0 \leq a_0,a_1,\ldots,a_{f-1} < p$, is a complete set of coset representatives for $\OO/\pp$.  Applying the construction of Example~\ref{ex:TBRBO} we may construct a {\TBRBO} sequence whose first $q=p^f$ terms are precisely the elements of $S$, with the following term being $p$.  Then, applying Corollary~\ref{cor:isoBern} shows that the transformation defined by the polynomial
\[ \frac{1}{p} \prod_{0 \leq a_0,a_1,\ldots,a_{f-1} < p} (x-a_0-a_1 \zeta - \cdots - a_{f-1}\zeta^{f-1}) \] is isometrically Bernoulli for $r = |p|$.
\end{example}

\section{Bernoulli maps on $\oh{\ZZ}$}\label{sec:bernZhat}
Define, as usual,
\[ \oh{\ZZ} = \ilim_{n,\divides} \ZZ/n\ZZ \isom \prod_{p} \ZZ_p. \]
We briefly note that the results of this section allow us to produce examples of maps $\NN \to \ZZ$ which extend to Bernoulli maps on $\ZZ_p$ for each $p$, and hence to a Bernoulli map on $\oh{\ZZ}$.  More explicitly, we obtain the following Proposition:
\begin{prop}
Let $a_0, a_1, \ldots$ be a sequence of integers satisfying the following conditions for each rational prime:
\begin{enumerate}
\item $|a_k|_p = 1$ for $k=p$;
\item $|a_k|_p < p^{-\lfloor \log_p k \rfloor}$ for $k > p$.
\end{enumerate}

Define $f: \NN \to \ZZ$ by \[ f(n) = \sum_{k=0}^n a_k {n \choose k}.\] Then, $f$ extends to an isometrically Bernoulli transformation $f: \ZZ_p \to \ZZ_p$ for each prime $p$.
\end{prop}
\begin{proof}
For each prime $p$, note that the quantity $p^{\lfloor \log_p k \rfloor} |a_k|_p$ attains its maximum for $k=p$ (and for no other $k$) and that $|a_k|_p = 1$.  Then, the result is immediate by Corollary~\ref{cor:isoBern}.
\end{proof}

\begin{example}
Let $f: \NN \to \ZZ$ be defined by
\[ f(n) = \sum_{p \leq n} \prod_{p' < p} {p'}^{1+ \lfloor \log_{p'} p \rfloor} {n \choose p} \]
where the summation is over primes bounded by $n$, and the product over primes bounded by $p$.  In the notation of the Proposition, we have
\[ a_p = \prod_{p' < p} {p'}^{1+\lfloor \log_{p'} p \rfloor}, \] and $a_k =0 $ for $k$ not a prime.  So, for each prime $p$ it is the case that $|a_p|_p = 1$.  Moreover, for $k >p$ we see that $|a_k|_p \leq p^{-1-\lfloor\log_{p} k \rfloor} < p^{-\lfloor \log_p k \rfloor}$.  So, the conditions of the Proposition are satisfied, and $f$ extends to an isometrically Bernoulli transformation $\ZZ_p \to \ZZ_p$ for each prime $p$.
\end{example}

\begin{example}
Let $f: \NN \to \ZZ$ be defined by
\[ f(n) = \sum_{k \leq n} {(k-1)!}^k {n \choose k}. \]
That is, we set $a_k = {(k-1)!}^k$.  Then, $|a_p|_p = |{(p-1)!}^p|_p = 1$.  And, for $k > p$ we see that certainly $|a_k|_p = |{(k-1)!}^{k-1}|_p < p^{-\lfloor \log_p k \rfloor}$.  So, the conditions of the Proposition are again satisfied.
\end{example}

\section{Polynomial almost Bernoulli maps on $\ZZ_p$}\label{sec:polAlmostBern}
An important condition shared by isometrically Bernoulli polynomials maps is that their derivatives must have constant valuation.  We may use this observation to come up with a class of interesting non-examples.  Our non-examples will be polynomial maps whose derivatives have constant valuation, but that are not measure-preserving.

\begin{prop}\label{prop:binomConstDeriv}
Given $n \in \NN$, the map $f: \ZZ_p \to \ZZ_p$ defined by \[ f(x) = {x \choose n} \] satisfies $|f'(x)| = C$ for some $C \in \VV$ and for all $x \in \ZZ_p$ if and only if $n = a p^\ell$, with $1 \leq a < p$ and $\ell \in \ZZ_{\geq 0}$, and \[ \frac{1}{u} + \ldots + \frac{1}{u+a-1} \not\equiv 0 \pmod p \] for each $u \in \{1,2,\ldots,p-a\}$, where inverses mod.~$p$ are taken in $\FF_p^*$.
\end{prop}
\begin{proof}
Write $n = a p^\ell + r$, $a < p$, $r < p^\ell$.  We first wish to show that $r=0$.  We may compute \[ \left| f'(a) \right| = \left|\binom{n-1}{a}\right| / \left|n\right| \qquad \text{for $0 \leq a \leq n-1$}.\]  If $r \geq 1$, then $(1+t)^{r-1} \equiv (1+t^{p^\ell})^{a} (1+t)^{r-1} \pmod p$ has no $t^{p^\ell-1}$ term, so we see that $|f'(p-1)| < |f'(0)|$.  Thus, $r=0$.  Note that $|f'(x)|$ is constant on $\ZZ_p$ iff it is constant on $\ZZ$, since $\ZZ \subset \ZZ_p$ is dense.  For $x \in \ZZ$, we observe that the set $\{ x, x-1, \ldots, x-n+1\}$ contains precisely $n/p^k$ terms divisible by $p^k$ for $k \leq \ell$, and either zero or one term divisible by some higher power of $p$.  By the strong triangle inequality, we note that $|f'(x)| \leq p^\ell$, with equality if there is a term divisible by $p^{\ell+1}$.  If no term is divisible by $p^{\ell+1}$, then let $u p^\ell, (u+1) p^\ell, \ldots, (u+a-1) p^\ell$ be the terms divisible by $p^\ell$ and observe that $u, u+1, \ldots, u+a-1$ are coprime to $p$.  An easy computation shows that $|f'(x)| = p^\ell$ if and only if 
\[ \sum_{i=0}^{a-1} (u+i)^{-1} \not\equiv 0 \pmod p.  \qedhere\]
\end{proof}

\begin{remark} Suppose the hypotheses of the Proposition hold.  Then, we may compute the value of $f(x) \pmod p$ by performing a careful but easy computation involving cancelling corresponding powers in $x(x-1)\cdots(x-n+1)$ (henceforth, ``the numerator'') and $n!$.  The only terms which are not obviously matched are those correspoding to terms divisible by $p^{\ell}$.  Suppose $u p^\ell, \ldots, (u+a-1) p^\ell$ are the terms in the numerator divisible by $p^\ell$, where now we need not assume that $u, \ldots, u+a-1$ are coprime to $p$.  Then,
\[ f(x) \equiv \frac{\prod_{i=0}^{a-1} (u+i)}{a!} \pmod p . \]
\end{remark}

The simplest family of maps satisfying the hypotheses of Prop.~\ref{prop:binomConstDeriv} is that in the following Corollary:
\begin{corollary}
Let $p \geq 3$, then $n = (p-2) p^\ell$ satisfies the conditions of Prop.~\ref{prop:binomConstDeriv} and so $|f'(x)| = p^\ell$ for all $x \in \ZZ_p$.
\end{corollary}
\begin{proof}
We may verify the last condition of Prop.~\ref{prop:binomConstDeriv} by an easy computation.
\end{proof}

\begin{example} Suppose $p \geq 3$, $\ell \geq 1$, and $n = (p-2) p^\ell$. Define a transformation $f: \ZZ_p \to \ZZ_p$ by \[ f(x) = {x \choose n}.\]   By the Corollary, we have that $|f'(x)| = p^\ell$ for all $x \in \ZZ_p$.  Mimicking the proof of Lemma~\ref{lem:binomBounds} we can show that $f$ is {\maM} for $r = p^{-(\ell+1)}$.   Moreover, since $(p-1)! \equiv -1 \pmod p$, the computation of the previous Remark yields
\[ f(x) \pmod p \equiv \begin{cases} 0 & \text{if $x \equiv 0, 1, \ldots, n-1 \pmod{p^{\ell+1}}$} \\ 1 & \text{if $x \equiv n, n+1, \ldots, n+p^\ell-1 \pmod{p^{\ell+1}}$} \\ -1 & \text{if $x \equiv n + p^\ell, \ldots, p^{\ell+1}-1 \pmod{p^{\ell+1}}$} \end{cases} \]  Identifying  $i \in \{0,1,\ldots,p^{\ell+1}-1\}$ with $i + p^{\ell+1} \ZZ_p$, the associated transition matrix for $f$ is thus
\[ A = (A(i,j))_{0 \leq i,j<p^{\ell+1}} = \begin{cases}\frac{p-2}{p} & \text{if $i \equiv 0, 1, \ldots, n-1 \pmod{p^{\ell+1}}$ and $j \equiv 0 \pmod p$} \\ \frac{1}{p} & \text{if $i \equiv  n, n+1, \ldots, n+p^\ell-1 \pmod{p^{\ell+1}}$ and $j \equiv 1 \pmod p$} \\ \frac{1}{p} & \text{if $i \equiv n + p^\ell, \ldots, p^{\ell+1}-1 \pmod{p^{\ell+1}}$ and $j \equiv -1 \pmod p$} \end{cases} \] 

So, if $p \neq 3$, we see that $f$ is not measure-preserving.  However, $A$ does have a left-eigenvector of eigenvalue $1$:
\[ \v =  (\v(i))_{0 \leq i < p^{\ell+1}} = \begin{cases} \frac{p-2}{p} & \text{if $i \equiv 0 \pmod p$} \\ \frac{1}{p} & \text{if $i \equiv \pm 1 \pmod p$} \end{cases} \]  This corresponds to an $f$-invariant measure $\sq{\mu}$ on $\ZZ_p$ defined on any $\mu$-measurable set $S \subset \ZZ_p$ by
\[ \sq{\mu}(S) = \frac{p-2}{p}\mu\left(S \cap B_{r}(0)\right) + \frac{1}{p} \mu\left(S \cap B_r(1)\right) + \frac{1}{p} \mu\left(S \cap B_r(-1)\right). \]  Then, the map $\Phi$ defined in Proposition~\ref{prop:maMPhi} gives a measurable isomorphism of $(\ZZ_p,\sq{\mu},f)$ with $(X_{A},\mu_{A,\v},T_{A})$, where the latter dynamical system is mixing Markov.
\end{example}

\bibliographystyle{amsalpha}
\bibliography{note}
\end{document}